\documentclass{amsart} 
\usepackage{amssymb}
\usepackage{amsmath}
\usepackage{amsfonts}

\begin{document}
\newtheorem{theo}{Theorem}[section]
\newtheorem{prop}[theo]{Proposition}
\newtheorem{lemma}[theo]{Lemma}
\newtheorem{exam}[theo]{Example}
\newtheorem{coro}[theo]{Corollary}
\theoremstyle{definition}
\newtheorem{defi}[theo]{Definition}
\newtheorem{rem}[theo]{Remark}


\newcommand{\Bb}{{\bf B}}
\newcommand{\Nb}{{\bf N}}
\newcommand{\Qb}{{\bf Q}}
\newcommand{\Rb}{{\bf R}}
\newcommand{\Zb}{{\bf Z}}
\newcommand{\Ac}{{\mathcal A}}
\newcommand{\Bc}{{\mathcal B}}
\newcommand{\Cc}{{\mathcal C}}
\newcommand{\Dc}{{\mathcal D}}
\newcommand{\Fc}{{\mathcal F}}
\newcommand{\Ic}{{\mathcal I}}
\newcommand{\Jc}{{\mathcal J}}
\newcommand{\Lc}{{\mathcal L}}
\newcommand{\Oc}{{\mathcal O}}
\newcommand{\Pc}{{\mathcal P}}
\newcommand{\Sc}{{\mathcal S}}
\newcommand{\Tc}{{\mathcal T}}
\newcommand{\Uc}{{\mathcal U}}
\newcommand{\Vc}{{\mathcal V}}

\newcommand{\ax}{{\rm ax}}
\newcommand{\Acc}{{\rm Acc}}
\newcommand{\Act}{{\rm Act}}
\newcommand{\ded}{{\rm ded}}
\newcommand{\Gm}{{$\Gamma_0$}}
\newcommand{\ID}{{${\rm ID}_1^i(\Oc)$}}
\newcommand{\PAP}{{${\rm PA}(P)$}}
\newcommand{\ACA}{{${\rm ACA}^i$}}
\newcommand{\RefP}{{${\rm Ref}^*({\rm PA}(P))$}}
\newcommand{\RefS}{{${\rm Ref}^*({\rm S}(P))$}}
\newcommand{\Rfn}{{\rm Rfn}}
\newcommand{\tar}{{\rm Tarski}}
\newcommand{\UNFA}{{${\mathcal U}({\rm NFA})$}}

\author{Nik Weaver}

\title [Constructive truth and circularity]
       {Constructive truth and circularity}

\address {Department of Mathematics\\
          Washington University in Saint Louis\\
          Saint Louis, MO 63130}

\email {nweaver@math.wustl.edu}

\date{\em April 13, 2010}

\begin{abstract}
We propose a constructive interpretation of truth which resolves
the standard semantic paradoxes.
\end{abstract}

\maketitle


\section{Heuristic concepts}
In intuitionism the law of excluded middle (LEM) is not accepted because the
assignment of truth values to sentences is seen as a kind of open-ended
process. Although the validity of any purported proof is supposed to be
decidable, the truth value of a given sentence may not be decidable because
one is not able to search through the infinite set of all potential proofs.
Thus the failure of LEM is related to the intuitionistic rejection of a
completed infinity.

Of course the validity of any proof within a given formal system is decidable.
But whether validity can really be considered a decidable property of proofs
broadly understood, outside of any particular formal system, is debatable.
The analogous claim is certainly not true of definitions. Indeed, suppose
we could decide whether any given finite string of words constructively
defines a natural number. Then in principle we would be able to unambiguously
determine which numbers are constructively defined by a string of
ten words by systematically examining all ten word long strings, and
consequently ``the smallest natural number not constructively definable
in ten words'' would be a valid constructive definition, which leads to a
contradiction.

We should conclude from this that definability is open-ended, but not in
the way intuitionists suppose truth to be open-ended, i.e., not merely
because one is unable to exhaustively search some infinite set. Rather, it
is open-ended in the sense that given any well-defined class of accepted
definitions we can always produce a new definition outside the family that
we would also accept. I will say that ``valid definition'' is a
{\it heuristic} concept \cite{W1}.
This is different from Dummett's notion of an indefinitely extensible
concept since he takes concepts to be decidable (\cite{Dum}, p.\ 441).

According to Troelstra it is ``natural'' to assume that the relation
``$c$ is a proof of $A$'' is decidable, and besides ``if we are in doubt
whether a construction $c$ proves $A$, then apparently $c$ does not prove
$A$ for us'' (\cite{Tro}, p.\ 7). But an identical argument could be
made in support of the claim that validity of definitions is decidable
(namely: if we are in doubt whether $c$ constructively defines a number
$n$, then $c$ does not constructively define $n$ for us). It
is not a good argument because it assumes that we can decide whether there
is any doubt about whether $c$ proves $A$. To the contrary, incompleteness
phenomena suggest that the general concept of a valid proof outside of any
particular formal system is not decidable. For if we can
accept, say, Heyting arithmetic (HA), then we can also accept a standard
arithmetization of the assertion that HA is consistent. This leads to a
stronger system whose validity we can also accept, and this process can be
iterated indefinitely, even transfinitely \cite {Tur, Fef}. It is hard
to see how the validity of proofs arbitrarily far up this hierarchy
could be decidable in any intuitionistically meaningful sense. Indeed, how
far we can get up the hierarchy depends on our ability to diagnose which
recursive total orderings of $\omega$ are well-founded, which is surely a
condition that is {\it not} intuitionistically decidable.

This example would not apply to a set-theoretic platonist who believes
that all sentences of first order number theory have well-defined
truth values. From his point of view, the use of proof in number theory
is necessary only because the infinite computations which could in
principle mechanically determine the truth value of any such sentence
are not available to us in practice. But this would not be the case
for sentences which quantify over proper classes; since there is no way,
even in principle, that one could perform an exhaustive search over a
proper class, for such sentences deductive proof reasserts itself as the
only means by which truth values can be known. Moreover, the set
theorist should also admit the meaningfulness of infinite formulas of
arbitrary cardinality. The problem that he then faces in deciding which
set theoretic axioms to accept is quite analogous to the intuitionist's
difficulty with number theory: any set of accepted axioms can always be
extended further. Specifically, given any set of axioms and deduction
rules one can assert the existence of a cardinal larger than all
cardinals whose existence is provable on that basis. So for the set
theorist too there does not seem to be any clear sense in which it is
decidable what would count as a valid proof. The collection of all true
infinite formulas is arguably a well-defined class, but membership in this
class is not decidable. Probably something analogous could be said in
regard to practically any coherent philosophy of mathematics.

I claim that {\it valid proof} should be understood as a heuristic concept,
and that this is independent of the question of the cogency of a completed
infinity (something I do in fact accept).

\section{Constructive reasoning}
Constructively, to assert that $A$ is true is to assert that $A$ has a
proof. (I discuss classical truth in section 9.) Whether we accept, e.g.,
infinitely long proofs, is not important here. All that matters is that
we understand truth in terms of being provable in some sense. The logical
symbols are interpreted as follows:
\begin{enumerate}
\item
A proof of $A \wedge B$ is a proof of $A$ and a proof of $B$.
\item
A proof of $A \vee B$ is a proof of $A$ or a proof of $B$.
\item
A proof of $A \to B$ is a construction that will convert any
proof of $A$ into a proof of $B$.
\item
A proof of $(\forall x)A(x)$ is a proof of $A(x)$ for arbitrary $x$.
\item
A proof of $(\exists x)A(x)$ is a proof of $A(x)$ for some $x$.
\end{enumerate}
What constitutes a proof of an atomic formula will vary depending on the
formal system in question. We interpret ``not $A$'' as an abbreviation of
the formula $A \to \bot$ where $\bot$ is some special atomic formula which
is false; for example, in arithmetical systems we may take $\bot$ to be
``$0 = 1$''.

The basic rules of natural deduction directly express the meanings of the
logical symbols:
\begin{enumerate}
\item
Given $A$ and $B$ deduce $A \wedge B$; given $A \wedge B$ deduce $A$ and
$B$.
\item
Given either $A$ or $B$ deduce $A \vee B$; given $A \vee B$, a proof of
$C$ from $A$, and a proof of $C$ from $B$, deduce $C$.
\item
Given a proof of $B$ from $A$ deduce $A \to B$; given $A$ and $A \to B$
deduce $B$.
\item
Given $A(x)$ deduce $(\forall x)A(x)$; if the term $t$ is free for $x$, given
$(\forall x)A(x)$ deduce $A(t)$.
\item
If the term $t$ is free for $x$, given $A(t)$ deduce $(\exists x)A(x)$; if
$y$ does not occur freely in $B$, given $(\exists x)A(x)$ and a proof
of $B$ from $A(y)$ deduce $B$.
\end{enumerate}
(See \cite{Pra} for a more precise, formal treatment of natural deduction.)

For example, our ability to deduce $A \wedge B$ from $A$ and $B$, and vice
versa, reflects the fact that having a proof of $A \wedge B$ is the same as
having a proof of $A$ and a proof of $B$. In the case of implication, we
argue as follows. If we have a proof $P$ of $B$ from $A$, then we can convert
any proof of $A$ into a proof of $B$ by concatenating it with $P$. Thus, we
have a construction that converts any proof of $A$ into a proof of $B$, i.e.,
we have a proof of $A \to B$. Conversely, if we have a proof of $A$ and a
construction that converts any proof of $A$ into a proof of $B$ then we have
a proof of $B$.

The deduction rules given above characterize Johansson's {\it minimal
logic} \cite{Joh}. It differs from intuitionistic logic in lacking ex
falso quodlibet. Minimal logic is appropriate for reasoning about
heuristic concepts.

\section{Ex falso quodlibet}
The ex falso principle states that anything follows from a falsehood:
$$\bot \to A.$$
The best universal
intuitionistic justification of this principle is that $\bot$ has
no proof, so vacuously we can convert any proof of $\bot$ into a proof
of $A$ \cite{vD}. However, if the notion of proof is heuristic then this
justification is defective. {\it The claim that $\bot$ has no proof
depends on the assumption that constructive reasoning is sound, but
if constructive reasoning is to include the ex falso principle --- whose
soundness is justified on the basis of that claim --- then this assumption
is circular.} In order to assure ourselves that constructive reasoning is
sound we must in particular establish that ex falso is sound, but showing
that ex falso is sound requires that we already know constructive reasoning
to be sound.

Another way to make the point is to say that we cannot affirm that
$\bot$ has no proof until we possess a clear concept of ``valid
proof'', and until we have determined whether ex falso is a legitimate
law to be used in proofs we apparently do not have a completely
clear concept of ``valid proof''.

This objection has no force if we take the validity of any proof to be
decidable. For then ``valid proof'' is a fixed, definite concept and it
is simply a fact that $\bot$ has no valid proof (though there is still
something dubious about the fact that our ability to decide that the proof
of ex falso is valid hinges on the premise that the validity of any proof
is decidable; this could be the basis of another argument that proof
validity cannot be decidable). But if proof is heuristic there is a
genuine circularity here because the act of adopting
the ex falso principle expands our notion of valid proof and according
to the justification given above the legitimacy of this move depends on the
correctness not just of previously accepted proofs, but also of newly
accepted proofs which themselves utilize ex falso.

One might suppose that this kind of circularity is benign, in the same
way that a sentence that asserts its own truth is benignly circular. But
again incompleteness phenomena show this to be a mistake. It is easy using
G\"odelian self-reference techniques to write down a formula $A$ in the
language of first order number theory that in a standard way
arithmetically expresses the
assertion that the formal system PA + $A$ consisting of the Peano axioms
plus $A$ is consistent. The formula $A$ embodies the same kind of
self-affirming circularity as the ex falso principle, but it is not benign:
we know from the second incompleteness theorem that PA + $A$ is not
consistent, and hence $A$ is false (indeed, provably false in PA).

Of course it is possible to justify ex falso in various (in practice
nearly all) situations. In settings where every statement has a definite
truth value it can be justified by interpreting implication classically
and appealing to truth tables. Or, since all we need is for
$\bot$ not to be provable, a proof of the equiconsistency of the
formal system in question with the same system minus the ex falso axiom
scheme may be enough to remove the vicious circle (and such a proof will
often be easy). Ex falso can also be justified for any formal system in
which we take $\bot$ to be ``$0 = 1$'' (or something similar) and in which
every atomic formula satisfies LEM; see Section 2.3 of \cite{W2}.

The ex falso principle is broadly valid, but it apparently does not have
a good universal non-circular justification.

\section{Inferring $A$ from ``$A$ is true''}
The inference `` `$A$ is true' implies $A$'' is, interpreted constructively,
similarly circular. For this statement to be true it must have a proof,
which would be a construction that converts any proof of ``$A$ is true''
into a proof of $A$. That is, we require a construction that converts any
proof that $A$ has a proof into a proof of $A$.

Since we are reasoning constructively, any proof that there exists a proof
of $A$ should, in principle, actually produce a proof of $A$. So we ought to
be able to convert a proof $P$ that $A$ has a proof into a proof $Q$ of $A$
simply by executing $P$. But just as in the case of ex falso quodlibet, the
success of this procedure depends on the assumption that constructive proofs
are sound. In order to justify `` `$A$ is true' implies $A$'' we must know
beforehand that all proofs that $A$ has a proof, {\it including proofs that
make use of the inference `` `$A$ is true' implies $A$''}, actually succeed
in producing a proof of $A$. As before, this is circular.

The inference ``$A$ implies `$A$ is true' '' is not constructively
problematic. Given a
proof $P$ of $A$, we can prove that $A$ has a proof by exhibiting $P$.
Thus we can convert any proof of $A$ into a proof that $A$ has a proof.
No assumption that all proofs are valid is involved. The reader is also
invited to refer back to Section 2 and assure himself that none of the
rules of natural deduction for minimal logic depend on assuming the
global validity of all proofs.

This last point perhaps needs emphasis. There is no circularity
in the other deduction rules. For example, the rule ``deduce $A \vee B$
from $A$'' requires that we be able to convert any proof of $A$ into a
proof of $A \vee B$, i.e., into a proof of $A$ or a proof of $B$. Of
course we can do this, since any proof of $A$ is already either a proof
of $A$ or a proof of $B$. No assumption about validity of proofs is
involved.

We must be careful to distinguish between the validity of an inference
and the validity of its conclusion. Certainly, if we are given an invalid
proof of $A$ we will produce an invalid proof of $A \vee B$;
in that sense the deduction rule does require the validity of all proofs
of $A$. But this does not affect the correctness of the inference.
Indeed, the inference ``if there is a $P$ then there is either a
$P$ or a $Q$'' is valid for {\it any} concepts $P$ and $Q$. The correctness
of this inference does not hinge on any particular property of these concepts.

(Modus ponens does require the validity of all proof constructions,
and since the notion of a proof is heuristic presumably the notion of
a construction that converts proofs into proofs is in general also
heuristic. However, the crucial point here is that adopting modus ponens,
while increasing our repertoire of proofs, does not increase our
repertoire of proof constructions, so there is no circularity in this.
The success of a given formal system typically depends only on the
validity of a limited class of proof constructions: we need to be able
to concatenate proofs for the sake of modus ponens, and axioms
that involve implication may presume additional specific proof
construction techniques. If we can affirm the validity of just these
kinds of proof constructions then we can justify modus ponens for
the system in question.)

\section{The liar paradox}
Consider the (strengthened) liar sentence $S$:

\centerline{This sentence is not true.}
\noindent It is apparently paradoxical because if $S$ is true then what it
asserts is the case; namely, $S$ is not true. But if $S$ is not true then
the assertion that it is not true is correct, and as this is just what $S$
asserts, $S$ must be true. In either case one is led to a contradiction.

A natural first reaction is to say that $S$ is neither true nor false,
but meaningless. This does not help because a sentence cannot be both
meaningless and true. Thus if it is meaningless then it is not true,
which, as we have just seen, leads to a contradiction.

The argument given above is not intuitionistically valid, since it
relies on the dichotomy ``$S$ is true or $S$ is not true''. But it can
easily be reformulated to avoid assuming LEM. Namely, we can argue as
before that {\it if} $S$ were true then it would not be true. Thus,
assuming that $S$ is true leads to a contradiction, which is exactly the
condition under which we may constructively affirm that $S$ is not true. So
without making any initial assumption about $S$ having or not having a truth
value, we can deduce that $S$ is not true. This then leads to a contradiction
just as before.

However, if truth is understood as heuristic, so that ``$A$ is true''
cannot be universally affirmed to entail $A$, then this argument fails. Given
the assumption that $S$ is true, for example, we cannot deduce that $S$ is not
true, only that we can prove that $S$ is not true. So contradictions are
blocked.

More formally, let $S$ be the formal liar proposition $S = T(\neg S)$
(``not-$S$ is true''). Then we have
$$S \to T(\neg S)$$
and since $A \to T(A)$ is generally valid we also have
$$S \to T(S).$$
We can therefore infer
$$S \to T(S \wedge \neg S)$$
and hence
$$S \to T(\bot).$$
But lacking $T(\bot) \to \bot$, we cannot infer $S \to \bot$, i.e., $\neg S$.

Starting from the hypothesis $\neg S$ also yields an informative result.
We have
$$\neg S \to T(\neg S)$$
(a special case of $A \to T(A)$) and
$$T(\neg S) \to S$$
(from the definition of $S$), so that
$$\neg S \to S.$$
Since also $\neg S \to \neg S$ this yields $\neg S \to \bot$, i.e.,
$\neg\neg S$. 

Informally: the liar proposition implies that there is a proof
of a contradiction ($S \to T(\bot)$), and additionally it is false
that the liar proposition is false ($\neg\neg S$). Formalizing the
liar sentence as $S' = \neg T(S')$ yields slightly different results;
in this case $\neg S'$ and $\neg\neg T(S')$ are provable. We can see
this by substituting $S'$ for $\neg S$ above.

If $T(A) \to A$ is not available one can still draw credible,
substantive conclusions about the liar sentence, but actual
contradictions are blocked.

\section{Axioms for self-applicative truth}
To demonstrate the consistency of the the kind of reasoning described above
we introduce a formal system.
The language is that of a propositional calculus with variables
$p_1, p_2, \ldots$. All formulas are built up from the propositional
variables and $\bot$ using $\wedge$, $\vee$, and $\to$. We fix an
enumeration $(A_i)$ of all formulas of the language
and interpret $p_i$ as the assertion that $A_i$ is true. For instance,
if $A_1 = \neg p_1$ then $p_1$ is a liar proposition. Or if $A_1 = p_2$
and $A_2 = \neg p_1$ then $p_1$ asserts that $p_2$ is true and $p_2$
asserts that $\neg p_1$ is true. The assumption that the $A_i$ are distinct
is not essential and could be removed at the cost of minor complication.

For any formula $A_i$ let $T(A_i)$ denote the corresponding propositional
variable $p_i$. We adopt the usual axioms and deduction rules of the
minimal propositional calculus (i.e., the intuitionistic propositional
calculus minus ex falso), together with the axioms

\centerline{$A \to T(A)$}
\centerline{$T(A) \wedge T(B) \leftrightarrow T(A \wedge B)$}
\centerline{$T(A) \vee T(B) \to T(A \vee B)$}
\centerline{$T(A \vee B) \wedge T(A \to C) \wedge T(B \to C) \to T(C)$}
\centerline{$T(A) \wedge T(A \to B) \to T(B)$,}
\noindent for all formulas $A$, $B$, and $C$. These axioms can all
be straightforwardly justified on the constructive interpretation
of the logical connectives described in Section 2. (For example: if
we have a proof of $T(A) \wedge T(B)$ then we have a proof of $T(A)$
and a proof of $T(B)$. That is, we have a proof that there is a proof
of $A$ and a proof that there is a proof of $B$. Combining them yields a
single proof that there is a proof of $A$ and a proof of $B$, i.e., a
proof that there is a proof of $A \wedge B$, i.e., a
proof of $T(A \wedge B)$. This shows that we can convert any proof of
$T(A) \wedge T(B)$ into a proof of $T(A \wedge B)$, which means that we
have a proof of $T(A) \wedge T(B) \to T(A \wedge B)$.)
We call the resulting formal system HT (Heuristic Truth). (The
system HT of course depends on the choice of the enumerating sequence
$(A_i$).)

HT is trivially consistent; just take $p_i$ to be true for all $i$ and
evaluate the truth of the $A_i$ classically. Then all the axioms of HT come
out true. More interesting is the fact that it is consistent at the
truth level, meaning that $T(\bot)$ is not a theorem (and hence no
contradiction can be proven true, since $A \wedge \neg A \to \bot$
and therefore $T(A) \wedge T(\neg A) \to T(A \wedge \neg A) \to T(\bot)$).

\begin{theo}
$T(\bot)$ is not a theorem of HT.
\end{theo}

\begin{proof}
This kind of consistency cannot be proven using classical models.
Indeed, in the example of a simple liar sentence discussed in the
previous section we have (as we showed there) $S \to T(\bot)$ and
$\neg S \to \bot$, hence $\neg S \to T(\bot)$, so that
$$(S \vee \neg S) \to T(\bot).$$
If HT contains a liar proposition then any classical model which satisfies
the law of excluded middle will not witness consistency at the truth level.

We verify consistency by assigning truth values in stages.
Introduce a new propositional symbol $\top$ representing truthhood and
say that a formula $A$ is {\it reduced} if $A = \top$, $A = \bot$,
or $A$ contains no occurrences of $\top$ or $\bot$. For any formula
$A$ define its {\it reduction} $A'$ inductively on the complexity of
$A$ by: $A' = A$ if $A$ is already reduced, and otherwise
\begin{enumerate}
\item
$(\top \wedge A)' = (A \wedge \top)' = A$
\item
$(\top \vee A)' = (A \vee \top)' = \top$
\item
$(\bot \wedge A)' = (A \wedge \bot)' = \bot$
\item
$(\bot \vee A)' = (A \vee \bot)' = A$
\item
$(\top \to A)' = A$
\item
$(A \to \top)' = \top$
\item
$(\bot \to A)' = \top$
\item
$(A \to \bot)' = \bot$ if $A \neq \bot$
\end{enumerate}
for any reduced formula $A$.

We now construct a function $\tau$ from the set of formulas of HT to
$\{0,1\}$ such that (a) $\tau(A) = 1$ for every axiom $A$ of HT, (b)
the set of $A$ such that $\tau(A) = 1$ is closed under modus ponens,
and (c) $\tau(T(\bot)) = 0$. This will show that $T(\bot)$
is not provable in HT. In the first step of
the construction, for all $i$ let $A_i^1 = A_i'$ be the reduction of $A_i$.
Then for each $i$ such that $A_i^1 = \bot$ define $\tau(A_i)
= \tau(p_i) = 0$, and for each $i$ such that $A_i^1 = \top$ define
$\tau(A_i) = \tau(p_i) = 1$. In the second step of the construction
we eliminate all propositional variables $p_i$ on which $\tau$ was defined
in the preceding step. Thus, for each $i$ let $A_i^2$ be the reduction
of $A_i^1$ after all occurrences of any $p_j$ with $\tau(p_j) = 0$ are
replaced by $\bot$ and all occurrences of any $p_j$ with $\tau(p_j) = 1$
are replaced by $\top$. Then for each $i$ such that $A_i^2 = \bot$ define
$\tau(A_i) = \tau(p_i) = 0$, and for each $i$ such that $A_i^2 = \top$
define $\tau(A_i) = \tau(p_i) = 1$. Continue this procedure indefinitely.
Observe that if $A_i^k = \top$ or $\bot$ then $A_i^{k'} = A_i^k$ for all
$k' > k$, i.e., truth values stabilize.
If $A_i^k$ never reduces to $\top$ or $\bot$ for any value of $k$ then we
define $\tau(A_i) = \tau(p_i) = 1$. This completes the definition of $\tau$.

All that remains is to prove that the function $\tau$ has the properties
(a) -- (c) claimed above. This is straightforward but tedious. For example,
to show that the set of $A$ such that $\tau(A) = 1$ is deductively closed,
suppose $\tau(A_i) = \tau(A_i \to A_j) = 1$. Let $k_i$ be the smallest
integer such that $A_i^{k_i} = \top$ or $\bot$ (or set $k_i = \infty$ if
there is no integer with this property), and define $k_j$ similarly.
Suppose $k_i \leq k_j$. Since $A_i^{k_i} = \top$ we must have
$(A_i \to A_j)^{k_i} = A_j^{k_i}$ by condition (5), and therefore (since
$k_i \leq k_j$)
$A_j^{k_j} = (A_i \to A_j)^{k_j} = \top$. This shows that $\tau(A_j) = 1$.
If instead $k_i > k_j$ then we must have $\tau(A_j) = 1$ since
otherwise $A_j^{k_j} = \bot$ and hence $(A_i \to A_j)^{k_j} = \bot$ by
condition (8), contradicting our assumption that $\tau(A_i \to A_j) = 1$.
So in either case we conclude that $\tau(A_j) = 1$. All other parts of
the claim are proven similarly.
\end{proof}

\section{Other paradoxes}
Grelling's paradox can be treated in a similar way. Formally, let
$T(P,x)$ stand for ``$P$ is true of $x$'', i.e., the predicate $P$ is
true when all free variables are replaced by $x$. As with the unary
truth predicate we can affirm $P(x) \to T(P,x)$ but not conversely. The
``heterological'' predicate can be formalized as $H(P) = T(\neg P, P)$.
We can then deduce
\begin{eqnarray}
H(H) &\to& T(H,H) \wedge T(\neg H,H)\cr
&\to& T(H \wedge \neg H, H)\cr
&\to& T(\bot, H)\cr
&\to& T(\bot)\cr
\end{eqnarray}
(using the principle that $T(P,x) \to T(P)$ if $P$ has no free variables)
and
$$\neg H(H) \to T(\neg H, H) \to H(H),$$
which implies $\neg \neg H(H)$. One can say the same thing about the
assertion that heterological is heterological that one can say about
the liar proposition: $H(H)$ implies that a falsehood is provable, and
it is false that $H(H)$ is false. Formalizing ``heterological'' as
$H'(P) = \neg T(P,P)$ yields the slightly different results
$\neg H'(H')$ and $\neg\neg T(H', H')$.

Berry's definability paradox mentioned in Section 1 is more complicated
to formalize since we have to reason about both numbers and expressions.
The essential ingredient is a predicate $D(A,n)$, ``$A$ defines $n$'',
interpreted as asserting that $A$ is true of $n$ and of no other number.
It will satisfy the axiom
$$A(n) \wedge [m \neq n \to \neg A(m)]\quad \to \quad D(A,n),$$
but the reverse implication will be problematic: without
assuming the global validity of all proofs, we cannot convert a proof
that $A$ is true of $n$ into a proof of $A(n)$.

\section{Reasoning consistently}
Formal systems for reasoning about truth typically face the difficulty
that, while semantic paradoxes are not formally derivable, nonetheless
the system itself is understood and analyzed in a metasystem from
which the paradoxes have not been exorcised. For example, Barwise and
Etchemendy consider truth to be situational and claim
that ``the Liar paradox shows that we cannot
in general make statements about the universe of all facts'' (\cite{BE},
p.\ 173), an assertion that seems to discredit itself and might even
count as a novel sort of liar sentence. Kripke, favoring a single
self-applicative truth predicate over a Tarskian hierarchy of truth
predicates, correctly notes that according to his theory
``Liar sentences are {\it not true} in the object language $\ldots$ but
we are precluded from saying this in the object language'' and states
``it is certainly reasonable to suppose that it is really the metalanguage
predicate that expresses the `genuine' concept of truth for the closed-off
object language'' (\cite{Kri}, pp.\ 714-715).

This criticism does not apply to the present account. Our notion of
constructive truth is univocal: there is no distinction between the truth
that is reasoned about formally (say, in the system HT) and the truth that
is discussed in the metasystem. We can prove in the metasystem that
the liar proposition is not provable in HT, but this does not entail
that it is not provable in general and so leads us to no definitive
conclusion about its actual truth.

But on the present account is the liar sentence {\it true} or {\it not true}?
Since we are reasoning constructively, this is not a forced dichotomy.
Perhaps the most we can say is that {\it if we reason consistently}
it is not true, which can be formalized as $\neg T(\bot) \to \neg T(S)$.

Unfortunately, we can actually prove the stronger assertion $\neg\neg T(\bot)$,
 i.e., $\neg T(\bot) \to \bot$. (This follows easily from $S \to T(\bot)$ and
$\neg\neg S$, both proven earlier; it implies $\neg T(\bot) \to \neg T(S)$
because $\bot \to \neg T(S)$.) That is, we can convert any proof
that we reason consistently into a proof of $\bot$. This does not
conflict with Theorem 6.1 because HT does not capture all forms of
acceptable reasoning (in particular, it does not capture the reasoning
employed in the proof of Theorem 6.1). Therefore the mere consistency
of HT at the truth level does not entail $\neg T(\bot)$.

I am arguing that the notion of a valid proof is heuristic, and any partial
formalization will always be capable of extension. What if we guarantee
consistency by simply demanding that no extension is to be accepted unless
it is consistent? In other words, we make a decision to reject any
extension of our notion of valid proof that leads to a proof of $\bot$.
We may certainly accept this prescription as an informal
guide. But if we were to adopt it formally in some way that would allow us
to infer that we do reason consistently, then we could deduce the
liar sentence and this would lead to a contradiction. So as a formal
principle that includes itself in its scope, the preceding proposal
disqualifies itself; it is viciously circular. {\it The only hope we have
of reasoning consistently is not to adopt principles that let us deduce
that we reason consistently.}

\section{Classical truth}
It could be objected that constructive truth as it has been characterized
here is not a notion of truth at all, on the grounds that a minimal
requirement of any account of truth is that it should satisfy Tarski's
biconditional `` `$A$ is true' if and only if $A$''. The response is that
as long as
we reason correctly it will actually be the case that we can prove $A$ if and
only if we can prove $T(A)$. But explicitly asserting $A \leftrightarrow T(A)$
is not acceptable because it implies a global assurance that our reasoning is
sound which we are not able to supply and which is in fact viciously circular.

To the contrary, I claim that in settings that involve a self-applicative
truth predicate, constructive truth is the only reasonable ``genuine''
notion of truth.

The classical intuition about truth is something like: $A$ is true if and
only if what $A$ asserts is actually the case. But this condition only makes
sense as applied to sentences whose meaning is already defined.
If a sentence refers to the concept of truth, then what it
means for this sentence to be the case depends on how we define truth, so
it is patently circular to define in one stroke the truth of all such
sentences in terms of
whether what they assert is the case. We cannot assume that there is a
fact of the matter about whether an assertion is the case, when that
assertion is framed in terms of concepts which have yet to be defined.

There are only two obvious remedies for this definitional quandry, if we
want to have a classical notion of truth that applies to sentences which
themselves contain a truth predicate. One is Tarski's idea of setting up a
hierarchy of truth predicates; in this case the sentence that asserts of
itself that it is not true${}_n$ will either be rejected as syntactically
illegitimate, or, if syntactically accepted, will be evaluated as not
true${}_n$ but true${}_{n+1}$. The other is Kripke's idea of generating a
single self-applicative truth predicate via transfinite recursion. But in
this case there will always be a gap between what we have defined as
``true'' in the object system and what we can prove to be the case in the
metasystem. Any step we take to extend the truth predicate to cover new
sentences will inevitably also
have the effect of rendering meaningful other sentences, not yet covered
by the truth predicate, which refer to the truth of the sentences whose
truth values have just been settled. The moral of this observation is just
that self-applicative truth is heuristic and cannot be understood as a
fixed well-defined concept. {\it It must be understood constructively.}


\end{document}